\documentclass{amsart}

\usepackage{color}
\usepackage{epsfig}
\usepackage{amsthm}
\usepackage{amssymb}
\usepackage{amsmath}
\usepackage{amscd}
\usepackage{color}

\newtheorem {Theorem}  {Theorem}

\numberwithin{Theorem}{section}

\newtheorem {Lemma}[Theorem]  {Lemma}
\newtheorem {Proposition}[Theorem]{Proposition}
\theoremstyle{definition}
\newtheorem{Definition}[Theorem]{Definition}
\theoremstyle{remark}

\newtheorem {Corollary}[Theorem]{Corollary}

\expandafter\chardef\csname pre amssym.def
at\endcsname=\the\catcode`\@ \catcode`\@=11
\def\undefine#1{\let#1\undefined}
\def\newsymbol#1#2#3#4#5{\let\next@\relax
 \ifnum#2=\@ne\let\next@\msafam@\else
 \ifnum#2=\tw@\let\next@\msbfam@\fi\fi
 \mathchardef#1="#3\next@#4#5}
\def\mathhexbox@#1#2#3{\relax
 \ifmmode\mathpalette{}{\m@th\mathchar"#1#2#3}%
 \else\leavevmode\hbox{$\m@th\mathchar"#1#2#3$}\fi}
\def\hexnumber@#1{\ifcase#1 0\or 1\or 2\or 3\or 4\or 5\or 6\or 7\or 8\or
 9\or A\or B\or C\or D\or E\or F\fi}

\font\teneufm=eufm10 \font\seveneufm=eufm7 \font\fiveeufm=eufm5
\newfam\eufmfam
\textfont\eufmfam=\teneufm \scriptfont\eufmfam=\seveneufm
\scriptscriptfont\eufmfam=\fiveeufm

\catcode`\@=\csname pre amssym.def at\endcsname

\newcounter{remark}
\newenvironment{remark}
{\medskip \stepcounter{remark} \noindent \textit{Remark
\arabic{section}.\arabic{remark}.}}{\rm \cbdu}

\def\3n{\negthinspace \negthinspace \negthinspace }
\def\2n{\negthinspace \negthinspace }
\def\1n{\negthinspace }

\newcommand{\bg}{\begin{equation}}
\newcommand{\ed}{\end{equation}}
\newcommand{\bga}{\begin{eqnarray}}
\newcommand{\eda}{\end{eqnarray}}

\def\cbdu{\hfill{$\Box$}}

\def  \12  {{\frac{1}{2}}}

\def\bd{\begin{definition}}
\def\ede{\end{definition}}
\def\be{\begin{equation}}
\def\bel{\begin{equation}\label}
\def\ee{\end{equation}}
\def\bt{\begin{Theorem}}
\def\et{\end{Theorem}}
\def\bc{\begin{Corollary}}
\def\ec{\end{Corollary}}
\def\bl{\begin{Lemma}}
\def\el{\end{Lemma}}
\def\bp{\begin{Proposition}}
\def\ep{\end{Proposition}}

\def\br{\begin{remark}}
\def\er{\end{remark}}
\def\ba{\begin{array}}
\def\ea{\end{array}}

\def\bea{\begin{eqnarray}}
\def\eea{\end{eqnarray}}

\begin{document}

\title[Decay of Infinite energy solutions to Navier-Stokes equations]{On the Decay of Infinite Energy Solutions to the Navier-Stokes Equations in the Plane}

\author[C. Bjorland]{Clayton Bjorland}
\address[C. Bjorland]{Department of Mathematics, University of Texas, 1 University Station C1200, Austin, TX, 78712 }
\email{bjorland@math.utexas.edu}

\author[C. J. Niche]{C\'esar J. Niche}
\address[C.J. Niche]{Departamento de Matem\'atica Aplicada, Instituto de Matem\'atica, Universidade Federal do Rio de Janeiro, CEP 21941-909, Rio de Janeiro - RJ, Brasil.}
\email{cniche@im.ufrj.br}

\thanks{C. Bjorland was partially supported by DMS-0636586. C.J. Niche was partially supported by FAPERJ, by research funds from UFRJ and by PRONEX E-26/110.560/2010-APQ1, FAPERJ-CNPq.}

\date{\today}


\begin{abstract}
Infinite energy solutions to the Navier-Stokes equations in $\mathbb{R}^2$ may be constructed by decomposing the initial data into a finite energy piece and an infinite energy piece, which are then treated separately.  We prove that the finite energy part of such solutions is bounded for all time and decays algebraically in time when the same can be said of heat energy starting from the same data. As a consequence, we describe the asymptotic behavior of the infinite energy solutions.  Specifically, we consider the solutions of Gallagher and Planchon \cite{MR1891170} as well as solutions constructed from a ``radial energy decomposition''.  Our proof uses the Fourier Splitting technique of M. E. Schonbek.
\end{abstract}

\maketitle

\section{Introduction}

The purpose of this paper is to explore the large time energy decay in $\mathbb{R}^2$ of solutions to the system
\begin{align}
u_t+u\cdot\nabla u+\nabla p-\triangle u&=-u\cdot\nabla v-v\cdot\nabla u,\label{pde}\\
\nabla\cdot u&=\nabla\cdot v=0,\notag\\
u(0)&=u_0\in L^2(\mathbb{R}^2)\notag
\end{align}
where $u$ is the velocity of an incompressible fluid, $p$ is its pressure and  $v$ is a specified external vector field satisfying
\begin{align}
\|\nabla^\alpha v\|_{L^\eta(\mathbb{R}^2)}\leq C t^{-\frac{1}{2}-\frac{\alpha}{2} +\frac{1}{\eta}} \label{vass}
\end{align}
for $\alpha =0$ when $2< \eta< \infty$ and either $\alpha=0$ or $\alpha =1$ when $\eta=\infty$.
Such a system arises naturally when considering infinite energy solutions of the Navier-Stokes equation, which includes the case of ``rough'' initial data in the plane.  

Recall that the Navier-Stokes equations for an incompressible viscous fluid are
\begin{align}
w_t+w\cdot\nabla w+\nabla p-\triangle w&=0,\label{NS}\\
\nabla\cdot w=0, \ \ \ \ \ \ w(0)&=w_0\notag
\end{align}
where $w$ represents the velocity of the incompressible viscous fluid and $p$ its pressure.  The literature involving this equation is quite large and we mention quickly a few relevant results.  One of the first rigorous mathematical treatments of this system in the plane $\mathbb{R}^2$ was the work of Leray \cite{MR1555394} in which global existence of a unique solution corresponding to initial data in $w_0\in L^2(\mathbb{R}^2)$ was established.  In $\mathbb{R}^3$ questions of global existence and uniqueness are much more difficult and there are outstanding open problems even at the level of $L^2$ initial data.  In $\mathbb{R}^2$ however there has been much work dedicated to finding solutions with initial data in larger function spaces, for example see Gallagher and Planchon \cite{MR1891170}, Cottet \cite{MR853597},  Giga, Miyakawa and Osada \cite{MR1017289}, Koch and Tataru \cite{MR1808843}, Germain \cite{MR2200642}, and the references therein.  Particularly relevant to our discussion, in \cite{MR1891170} and \cite{MR2200642} the authors used a technique which involved separating the solution into a ``rough'' part and a finite energy part which satisfies (\ref{pde}). 

Formally, if initial data $w_0$ is decomposed as $w_0=v_0+u_0$ with $u_0\in L^2(\mathbb{R}^2)$ and if $v(t)$ solves (\ref{NS}) with initial data $v_0$, then a solution of (\ref{NS}) with data $w_0$ can be written as $w(t)=u(t)+v(t)$ where $u$ satisfies (\ref{pde}) with initial data $u_0$.   The energy decay theorem we prove indicates that the energy of solutions to (\ref{pde}), that is $\|u(t)\|_{L^2(\mathbb{R}^2)}$, remains bounded and decays algebraically when the same can be said of the corresponding heat energy.  In turn, this describes how $w(t)$ approaches $v(t)$ in the $L^2$ norm as time becomes large even though $w$ and $v$ need not be in $L^2$ individually.

The main result in this article is the following Theorem:

\bt
\label{maintheorem}
Let $u$ be a global solution to (\ref{pde}), with $v$ satisfying (\ref{vass}).  We assume:
\begin{itemize}
 \item[(i.)] For each $t_0>0$ there is a constant $C_{t_0}>0$ such that for all $T>t_0$,
\begin{align}
\sup_{t_) \leq t \leq T}\|u(t)\|_{L^2(\mathbb{R}^2)} \leq C_{t_0}(1+T)^{\frac{1}{2}}.\label{initialenergy}
\end{align}
 \item[(ii.)] For some $\gamma\in [0,1]$ we have $\|e^{\triangle t}u_0\|_2^2\leq C(1+t)^{-\gamma}$ where $e^{\triangle t}u_0$ denotes the solution to the heat equation with initial data $u_0$.
\end{itemize}
Then for every $t_0>0$ there exists a constant $\tilde{C}_{t_0}$ such that
\begin{align}
 \|u(t)\|_{L^2(\mathbb{R}^2)}^2\leq \tilde{C}_{t_0}(1+t)^{-\gamma}\notag
\end{align}
for all $t>t_0$.
\qed
\et

\begin{remark}
Assumption (i.) in the above theorem is the natural a priori energy estimate for (\ref{pde}), a formal proof is given in Subsection \ref{energysubsection}.  Assumption (ii.) takes into account the natural decay rate for heat energy starting from $u_0$.  The Theorem states that if the heat energy starting from $u_0$ decays like $(1+t)^{-\gamma}$ with $\gamma\in [0,1]$, then the solution $u(t)$ of (\ref{pde}) has the same energy decay rate.   This is natural, as the heat equation is the linear part of (\ref{pde}) and we do not expect  solutions to (\ref{pde}) to decay faster than this.  On the other hand the ``rough'' terms (the nonlinear terms containing $v$) can ``mix'' the solution and slow the energy decay.
\end{remark}

%

\begin{remark}
It is known that the heat energy decay rate is determined by the behavior of $u_0$ near the origin in Fourier Space.  For example, if $u_0 \in L^1(\mathbb{R}^2) \cap L^2(\mathbb{R}^2)$, so that $|\hat{u}_0(\xi)|<C$ near $|\xi|=0$, then  $\|e^{\triangle t}u_0\|_2^2 \leq (1+t)^{-1}$ and hence $\|u(t)\|_{L^2(\mathbb{R}^2)} \leq C(1+t)^{-\frac{1}{2}}$.  More detailed analysis may be found in Bjorland and Schonbek \cite{MR2493562}.   Although $u_0\in L^2(\mathbb{R}^2)$ implies $\|e^{\triangle t}\|_2\rightarrow 0$ as $t\rightarrow \infty$, the heat energy may not decay at an algebraic rate (i.e. $\gamma=0$). This allows us to construct solutions to (\ref{pde}) with arbitrarily slow decay by appropiately scaling the initial data and the external vector field, by using the same arguments as for the Navier-Stokes equations (for details on this case see Schonbek \cite{MR837929}).   

\end{remark}

\begin{remark}
The proof of Theorem \ref{maintheorem} is based on the Fourier Splitting method of M. E. Schonbek \cite{MR775190}, \cite{MR837929} introduced to study algebraic energy decay rates in parabolic equations.  
\end{remark}


\vspace{.25cm}

We now indicate how to use Theorem \ref{maintheorem} to better understand the large time behavior of infinite energy solutions to $2D$ Navier-Stokes solutions.   By an infinite energy solution we mean one belonging to one of the scale-invariant homogeneous Besov spaces $\dot{B}^{2/r - 1} _{r,q} (\mathbb{R}^2)$ which satisfy the chain of continuous embeddings

\be \label{eqn:besov-embeddings}
L^2(\mathbb{R}^2) \subset \dot{B}^{2/r - 1} _{r,q} (\mathbb{R}^2) \subset \dot{B}^{2/\tilde{r} - 1} _{\tilde{r},\tilde{q}} (\mathbb{R}^2) \subset BMO ^{-1} (\mathbb{R}^2) \subset \dot{B} ^{-1} _{\infty, \infty} (\mathbb{R}^2)
\ee
where $2 \leq r \leq \tilde{r} < \infty$ and $2 \leq q \leq \tilde{q} \leq \infty$.

Consider Navier-Stokes equations (\ref{NS}) with initial data $w_0\in\dot{B}^{2/r - 1} _{r,q} (\mathbb{R}^2)$, with $r, q < \infty$. In these spaces, Gallagher and Planchon \cite{MR1891170} proved existence of global solutions. To prove this result they decompose $w_0=v_0+u_0$ where $u_0\in L^2(\mathbb{R}^2)$ and  $v_0\in BMO^{-1}(\mathbb{R}^2)$ with small norm.  Starting from this small $v_0$ they construct a solution $v(t)$ of the Navier-Stokes equation using a fixed point argument which naturally satisfies (\ref{vass}) for $\eta\in [1,\infty]$.  Next they consider (\ref{pde}) and find a solution $u(t)$ using a fixed point theorem to obtain local existence then prove an a priori energy estimate to establish global existence.  The energy bounds used imply $\|u(t)\|_{L^2 (\mathbb{R}^2)} \leq Ct^{1/2}$, which is exactly Assumption (i.) in Theorem \ref{maintheorem}, though the authors leave open the question of finding better bounds on $u$.  An interpolation argument is then used to show $w(t)=v(t)+u(t)\in \dot{B}^{2/r - 1} _{r,q} (\mathbb{R}^2)$ for all time.  Using similar methods, Germain \cite{MR2200642} proved global existence of solutions for data in $VMO ^{-1} (\mathbb{R}^2) $ which is the closure of the Schwartz space in $BMO ^{-1} (\mathbb{R}^2)$.  Moreover, he proved that under some mild conditions on $r$ and $q$, Gallagher and Planchon's solutions with initial data in $\dot{B}^{2/r - 1} _{r,q} (\mathbb{R}^2)$ are such that

\begin{displaymath}
\lim _{t \to \infty} \Vert u(t) \Vert _{\dot{B}^{2/r - 1} _{r,q} (\mathbb{R}^2)} = 0.
\end{displaymath}


In this context we can use Theorem \ref{maintheorem} to prove that the ``finite energy part'' of the infinite energy solution decays algebraically when the same can be said of the corresponding heat equation. This is the content of the following corollary:

\begin{Corollary} \label{corollary-besov}
Let $w_0 \in \dot{B}^{2/r - 1} _{r,q} (\mathbb{R}^2)$, with $r, q < \infty$. Consider $w_0=v_0+u_0$, where $u_0\in L^2(\mathbb{R}^2)$ and  $v_0\in BMO^{-1}(\mathbb{R}^2)$ with small norm.  Let $v(t)$ and $w(t)$ be the solutions of (\ref{NS}) given in \cite{MR1891170} with initial data $v_0$ and $w_0$ respectively.  If $\Vert e^{t \Delta} u_0 \Vert _{L^2(\mathbb{R}^2)} \leq C (1 + t)^{- \gamma}$ for some $\gamma \in [0,1]$, then

\begin{displaymath}
\Vert w(t) - v(t) \Vert _{L^2(\mathbb{R}^2)} \leq C (1 + t)^{- \gamma}.
\end{displaymath}
\qed
\end{Corollary}

\begin{remark}
 In particular, for any $t_0>0$ we have $\|u\|_{L^\infty([t_0,\infty);L^2(\mathbb{R}^2))}<\infty$ which is stronger then the original energy estimate.
\end{remark}
\vspace{.25cm}


More classically, Theorem \ref{maintheorem} can be used to understand long time behavior of infinite energy solutions to the Navier-Stokes equations with finite local energy and integrable initial vorticity, that is $w_0\in L^2_{loc}(\mathbb{R}^2)$ and $\omega_0 = \nabla\times w_0\in L^1(\mathbb{R}^2)$. This initial data is a particular case of so-called vortex sheet initial data and it was used by DiPerna and Majda \cite{MR882068} to study approximate solution sequences for the Euler equation (see also \cite[Sec. 3.1.2]{MR1867882}). 

\begin{remark}
Initial data $w_0$ is of vortex sheet type if $w_0\in L^2_{loc}(\mathbb{R}^2)$ and $\omega_0 = \nabla\times w_0 \in \mathcal{M}(\mathbb{R}^2)$, where $\mathcal{M}(\mathbb{R}^2)$ is the space of nonnegative Radon measures. As for any $\omega _0  \in \mathcal{M}(\mathbb{R}^2)$ there exists a unique $w_0 \in \dot{B} ^1 _{1, \infty} (\mathbb{R}^2) \subset \dot{B} ^{2/r -1} _{r, \infty} (\mathbb{R}^2)$ such that $\omega _0 = curl \, w_0$ (see Corollary 4.4, Germain \cite{MR2200642}), then $w_0$ is in one of the infinite energy spaces in (\ref{eqn:besov-embeddings}).
\end{remark}


\begin{Definition}
An incompressible velocity field $w_0:\mathbb{R}^2\rightarrow \mathbb{R}^2$ has a ``radial energy decomposition'' if there exists a smooth radially symmetric vorticity $\bar{\omega}_0(|x|)$ such that
\begin{align}
 w_0(x)=u_0(x)+v_0(x),\notag\\
\int_{\mathbb{R}^2}|u_0(x)|^2\,dx <\infty,\notag
\end{align}
where $v_0$ is defined from $\bar{\omega}_0$ by the Biot-Savart law $v_0=K\ast \bar{\omega}_0$, for $K(x)=\frac{1}{2\pi}\frac{x^\perp}{|x|^2}$ the $2D$ Biot-Savart kernel.  The radial energy decomposition, which is not unique, is possible on the whole plane since $u_0\in L^2(\mathbb{R}^2)$ if and only if $\int_{\mathbb{R}^2} \nabla\times u_0\,dx =0$.  
\end{Definition}
We restrict our attention to  initial data with $w_0\in L^2_{loc}(\mathbb{R}^2)$ and $\omega_0 = \nabla\times w_0\in L^1(\mathbb{R}^2)$, because it can be split appropiately using the radial energy decomposition (see Lemma 3.2 in Majda and Bertozzi \cite{MR1867882}). Moreover, some of the estimates we use when working with initial vorticity in $L^1(\mathbb{R}^2)$ need not be available in  $\mathcal{M} (\mathbb{R}^2)$. Denote by $\bar{\omega} (x,t)$ the solution to the heat equation with initial data $\bar{\omega}_0$. As the initial data is radial, so is $\bar{\omega} (x,t)$, and it is a solution to the vorticity formulation of Navier-Stokes equation

\begin{align} \label{eqn:vorticity}
\partial_t\bar{\omega} +v\cdot\nabla \bar{\omega} &=\triangle \bar{\omega},  \\
v&=K\ast \bar{\omega}(t),\notag\\
\bar{\omega} _0 (x)& = \bar{\omega} (x,0). \notag
\end{align}
With $v=K\ast \bar{\omega}$ in hand we may then find the solution $u(t)$ to (\ref{pde}) starting from initial data $u_0$ using energy methods as outlined in \cite{MR1867882}, thus obtaining the solution $w(t)=v(t)+u(t)$ of the Navier-Stokes equation.  In Subsection \ref{vortex-sheet-bound} we show how $v(t)$ satisfies (\ref{vass}).  We have then the following Corollary:

\begin{Corollary} \label{corollary-vortexsheet}
Let $w(t)$ be a solution of the Navier-Stokes equation  with initial data $w_0\in L^2_{loc}(\mathbb{R}^2)$ such that $\omega_0 = \nabla\times w_0 \in L^1 (\mathbb{R}^2)$.  Let $w_0 = u_0 + v_0$ be a radial energy decomposition with $u_0 \in L^2 (\mathbb{R}^2)$, $\bar{\omega}_0$ a radial function, and $v_0=K\ast \bar{\omega}_0$. If $\| e^{ \Delta t} u_0 \|_{L^2(\mathbb{R}^2)}^2 \leq C (1 + t)^{- \gamma}$ for some $\gamma \in [0,1]$, then

\begin{displaymath}
\Vert w(t) - v(t) \Vert _{L^2(\mathbb{R}^2)}^2 \leq C (1 + t)^{- \gamma}
\end{displaymath}
where
\begin{align}
v(x,t)= \frac{x^\perp}{|x|^2}\int_0^r se^{\triangle s}\bar{\omega}_0(s)\,ds\notag
\end{align}
\qed
\end{Corollary}

\begin{remark}\label{remLp}
Using a far field calculation it can be shown that if $\nabla\times u_0$ has compact support then $u_0\in L^p(\mathbb{R}^2)$ for any $p\in (1,2]$ and $\gamma$ can be chosen to be any value in $[0,1)$.  This is demonstrated in Subsection \ref{vortex-sheet-bound}.
\end{remark}

\begin{remark}
For $\omega_0 \in L^1 (\mathbb{R}^2) $, Gallay and Wayne \cite{MR1912106}, \cite{MR2123378} have described the asymptotic behavior of solutions to the vorticity equation (\ref{eqn:vorticity}). In particular they prove

\begin{displaymath}
\lim _{t \to \infty} t^{\frac{1}{2} - \frac{1}{q}} \Vert w(t) - \frac{\alpha}{\sqrt{t}} V \left( \frac{\cdot}{\sqrt{t}} \right)\Vert _{L^q (\mathbb{R}^2)} = 0, \qquad 2 < q \leq \infty
\end{displaymath}
where $\alpha = \int _{\mathbb{R}^2} \bar{\omega} _0 \, dx$ and $V(\xi) = \frac{1}{2 \pi} \frac{\xi ^{\bot}}{|\xi|^2} \left( 1 - e^{-|\xi|^2/4}\right)$.  Our Corollary concerns the borderline case $q = 2$, but we show how the solution approaches a radial solution instead of the Oseen vortex
\begin{displaymath}
O(\xi, t) = \frac{\alpha}{\sqrt{t}} V \left( \frac{\xi}{\sqrt{t}} \right) = \frac{\alpha}{2 \pi} \frac{\xi ^{\bot}}{|\xi|^2} \left( 1 - e^{-|\xi|^2/4t}\right). 
\end{displaymath}
Note that the Oseen vortex is a solution to the Navier-Stokes equations (\ref{NS}) with initial data $w_0 (\xi) = \frac{1}{2 \pi} \frac{\xi ^{\bot}}{|\xi|^2}$, which is not in $L^2_{loc}(\mathbb{R}^2)$, but is in $\dot{B} ^{2/r -1} _{r, \infty} (\mathbb{R}^2)$ because is a homogeneous distribution of degree $-1$ (see, Cannone \cite[Lemma 3.3.2]{MR1688096}).
\end{remark}

\vspace{.25cm}

This articles is organized as follows. In the next Section we establish some basic properties of solutions to (\ref{pde}), including the a priori energy estimate.  In Section \ref{decaysection} we use the Fourier Splitting Method to prove Theorem \ref{maintheorem}.

\section{Preliminaries}


\subsection{A Priori Energy Estimate} \label{energysubsection}

We now establish an a priori energy estimate for solutions of (\ref{pde}) when $v$ satisfies (\ref{vass}) with $\alpha = 0$ and $\eta = \infty$.  This estimate is known in the literature but we record it here for completeness since it is one of the assumptions for Theorem \ref{maintheorem}.  It is straightforward to make it precise in the case of the radial energy decomposition mentioned in the Introduction (see \cite{MR1891170} for a rigorous argument in their setting). Formally, multiplying (\ref{pde}) by $u$ and then integrating by parts yields
\begin{align}
\frac{1}{2}\frac{d}{dt}\|u\|_{L^2(\mathbb{R}^2)}^2 + \|\nabla u\|_{L^2(\mathbb{R}^2)}^2 \leq |<u\cdot\nabla v,u>|\label{energystep1}
\end{align}
where we have introduced the notation $<u\cdot\nabla v,u>=\sum_i\int u\cdot\nabla v_i u_i\,dx$.  Fix $t_0>0$.  After integrating by parts and using H\"older's inequality, then (\ref{vass}) with $\alpha = 0$ and $\eta = \infty$ and then Cauchy's inequality,  we have for any $t>t_0$,
\begin{align}
|<u(t)\cdot\nabla v(t),u(t)>|&=|<u(t)\cdot\nabla u(t),v(t)>|\notag\\
&\leq C\|u(t)\|_{L^2(\mathbb{R}^2)}^2(1+t)^{-1} +\frac{1}{2}\|\nabla u(t)\|_{L^2(\mathbb{R}^2)}^2.\notag
\end{align}
In the above line the constant may depend on $t_0$.  Combining this estimate with (\ref{energystep1}) and then integrating from $t_0$ to $t$ yields
\begin{align}
\|u(t)\|_{L^2(\mathbb{R}^2)}^2 \leq C\int_{t_0}^t\|u(s)\|_{L^2(\mathbb{R}^2)}^2(1+s)^{-1}\,ds +\|u(t_0)\|_{L^2(\mathbb{R}^2)}^2.\notag
\end{align}
From here a Gronwall inequality gives
\begin{align}
\|u(t)\|_{L^2(\mathbb{R}^2)}^2 \leq C\|u(t_0)\|_{L^2(\mathbb{R}^2)}^2(1+t)\notag
\end{align}
which is (\ref{initialenergy}) of assumption (i.) in Theorem \ref{maintheorem}.

\subsection{Properties of solutions with the radial energy decomposition} \label{vortex-sheet-bound}
In this subsection we consider the Navier-Stokes equation with initial data $w_0\in L^2_{loc}(\mathbb{R}^2)$ such that $\omega_0 = \nabla \times w_0 \in L^1(\mathbb{R}^2)$.  As in the Introduction, consider the radial energy decomposition $w_0=u_0+v_0$ where $u_0\in L^2(\mathbb{R}^2)$ and $v_0$ is the velocity of the radial vorticity $\bar{\omega}_0$.  We first we prove our claim that $v=K\ast e^{\triangle t}\bar{\omega}_0$ satisfies the estimate (\ref{vass}) with $\alpha = 0$ and $\eta = \infty$. As $\bar{\omega} _0 \in L^1 (\mathbb{R}^2)$ we have by direct calculation  
\begin{displaymath}
\Vert e^{\triangle t}\bar{\omega}_0 \Vert _{L^p (\mathbb{R}^2)} \leq C t ^{- (1 - \frac{1}{p})}, \quad 1 \leq p \leq \infty.
\end{displaymath} 
To find the estimate on $v(t)$, the corresponding solution to the Navier-Stokes equations, we recall the following estimate on the Biot-Savart Kernel.
\begin{Lemma}
Let $\bar{\omega}_0\in L^p (\mathbb{R}^2) \cap L^q(\mathbb{R}^2)$ for $1 \leq p < 2 < q \leq \infty$ and let $0 < \alpha < 1$ be such that $\frac{1}{2} = \frac{\alpha}{p} + \frac{1 - \alpha}{q}$.  For $v=K\ast\bar{\omega}_0$ we have
\begin{displaymath}
\Vert v(t) \Vert _{L^{\infty}(\mathbb{R}^2)} \leq C \Vert \bar{\omega}_0 \Vert  _{L^p (\mathbb{R}^2)} ^{\alpha} \Vert \bar{\omega}_0 \Vert _{L^q (\mathbb{R}^2)} ^{1 - \alpha}.
\end{displaymath} 
\end{Lemma}
\begin{proof}
See \cite[Lemma 2.1]{MR1912106}.
\end{proof}
Combining the previous lemma with the above bound on $e^{\triangle t}\bar{\omega}_0$ we establish (\ref{vass}). As mentioned in the Introduction, if we further assume that $\tilde{\omega}_0 = \nabla\times u_0$ has compact support $B_R$ we can use a far field calculation to demonstrate $\|e^{\triangle t}u_0\|_2^2\leq C(1+t)^{-\gamma}$ for every $\gamma\in [0,1)$.  Indeed, if $y<R$ and $x>4R$ then the following geometric series converges:
\begin{align}
 \frac{1}{|x-y|}=\frac{1}{|x|^2}\sum_{k=0}^\infty\left( \frac{|y|^2}{|x|^2}-\frac{2x\cdot y}{|x|^2}\right)^k.\notag
\end{align}
Using $\int_{\mathbb{R}^2}\nabla\times u_0\,dx=0$ we find that for large $x$ 
\begin{align}
 u_0(x)&=\frac{1}{2\pi}\left(-\frac{1}{|x|^2}\int_{\mathbb{R}^2}y^\perp\tilde{\omega}_0(y)\,dy-\frac{x^\perp}{|x|^4}x\cdot \int_{\mathbb{R}^2}y\tilde{\omega}_0(y)\,dy+ O(|x|^{-3})\right)\notag
\end{align}
which implies that $u_0\in L^p(\mathbb{R}^2)$ for every $p\in (1,2]$.  For $q$ such that $\frac{1}{p}+\frac{1}{q}=\frac{3}{2}$ we bound
\begin{align}
\|e^{\triangle t}u_0\|_2 \leq \|\Phi (t)\|_{L^q(\mathbb{R}^2)}\|u_0\|_{L^q(\mathbb{R}^2)}\notag
\end{align}
where $\Phi (t)$ is the $2D$ heat kernel.  As  $\|\Phi (t)\|_{L^q(\mathbb{R}^2)} \leq Ct^{\frac{1}{q}-1}$ it must be that $\|e^{\triangle t}u_0\|_2\leq Ct^{\frac{1}{2}-\frac{1}{p}}$.  Since $p\in (1,2]$ we have $\|e^{\triangle t}u_0\|_2^2\leq C(1+t)^{-\gamma}$ for every $\gamma\in [0,1)$.

\section{Decay} \label{decaysection}

In this section we prove Theorem \ref{maintheorem} using the Fourier Splitting Method of M. E. Schonbek. In our proof we also incorporate a Gronwall-type trick used by Zhang \cite{MR1312702}. Here we proceed formally but note the argument can be made rigorous using an approximating sequence of solutions.  This would be argued similar to the proof of the energy inequality (\ref{initialenergy}) in \cite{MR1891170} or similar to \cite{MR775190} in the more classical radial energy decomposition case.  We start with frequency bounds. Applying Duhamel's formula in Fourier space and a simple integral inequality to (\ref{pde}) yields
\begin{align}
|\hat{u}|\leq e^{-|\xi|^2t}|\hat{u}_0|+\int_0^te^{-|\xi|^2(t-s)}|\xi|\left(|\widehat{v\otimes u}|+|\widehat{u\otimes u}| +|\widehat{u\otimes v}|+|\hat{p}|\right)\,ds.\label{feq}
\end{align}
Taking divergence of (\ref{pde}) and then using the symmetry of the tensor we find that $|\hat{p}|\leq 2|\widehat{v\otimes u}|+|\widehat{u\otimes u}|$, so we obtain the bound
\begin{align}
|\hat{u}|\leq e^{-|\xi|^2t}|\hat{u}_0|+2\int_0^te^{-|\xi|^2(t-s)}|\xi|\left(|\widehat{v\otimes u}|+|\widehat{u\otimes u}|\right)\,ds.\notag
\end{align}
Using now H\"older's inequality with the estimate (\ref{vass}) ($\eta=\infty$ and $\alpha=1$) gives
\begin{align}
|<u\cdot\nabla v,u>|\leq Ct^{-1}\|u(t)\|_{L^2(\mathbb{R}^2)}^2\notag
\end{align}
so that after multiplying the PDE by $u$ and integrating by parts we have
\begin{align}
\frac{1}{2}\frac{d}{dt}\|u\|_{L^2(\mathbb{R}^2)}^2+\|\nabla u\|_{L^2(\mathbb{R}^2)}^2 \leq Ct^{-1}\|u\|_{L^2(\mathbb{R}^2)}^2.\notag
\end{align}
We fix $t_0>0$ and now consider the inequality for $t>t_0>0$ so that $t^{-1}<(1+t_0^{-1})(1+t)^{-1}$ and 
\begin{align}
Ct^{-1}\|u\|_{L^2(\mathbb{R}^2)}^2 \leq C_0(t+1)^{-1}\|u\|_{L^2(\mathbb{R}^2)}^2\notag,
\end{align}
where $C_0$ contains the term $(1+t_0^{-1})$.  We now apply a Fourier Splitting argument around a ball with radius $r (t)>0$, where $r(t)$ is to be determined later.  After observing that
\begin{align}
r^2\|u\|^2_{L^2(\mathbb{R}^2)}-r^2\int_{B(r)}|\hat{u}(s)|^2\,d\xi \leq \|\nabla u\|_{L^2(\mathbb{R}^2)}\notag
\end{align}
we find that
\begin{align}
\frac{1}{2}\frac{d}{dt}\|u\|_{L^2(\mathbb{R}^2)}^2+ (r^2-C_0(t+1)^{-1})\|u\|_{L^2(\mathbb{R}^2)}^2\leq r^2\int_{B(r)}|\hat{u}(s)|^2\,d\xi\label{energysplit}
\end{align}
for $t>t_0$. 

In the case where (\ref{vass}) does not hold for $\eta=\infty, \alpha=1$ we can instead use (\ref{vass})  with $\eta=\infty,  \alpha=0$ as mentioned in the Introduction. After integration by parts and using Cauchy's inequality we obtain the bound
\begin{align}
|<u\cdot\nabla v,u>|=|<u\cdot\nabla u,v>|&\leq \|u\|_{L^2(\mathbb{R}^2)}\|\nabla u\|_{L^2(\mathbb{R}^2)}\|v\|_{L^\infty(\mathbb{R}^2)}\notag\\
&\leq C(1+t)^{-1}\|u(t)\|_{L^2(\mathbb{R}^2)}^2 +\frac{1}{2}\|\nabla u\|_{L^2(\mathbb{R}^2)}^2.\notag
\end{align}
Considering again a fixed $t_0>0$  we again arrive at (\ref{energysplit}) but with different constants which will have no impact on the following arguments.  Thus, we can say that Theorem \ref{maintheorem} holds for these two estimates on $v$, which is what we use as hypotheses in our Corollaries \ref{corollary-besov} and \ref{corollary-vortexsheet}.

Now we estimate the right hand side of (\ref{energysplit}):
\begin{align}
\int_{B(r)}|\hat{u}(s)|^2\,d\xi&\leq \int_{B(r)}e^{-2|\xi|^2t}|\hat{u}_0|^2\,d\xi\notag\\
&\ \ \ \ \ \ \ \ + \int_{B(r)}\left(\int_0^te^{-|\xi|^2(t-s)}|\xi|(|\widehat{v\otimes u}|+|\widehat{u\otimes u}|)\,ds\right)^2\,d\xi\notag\\
&:= I(t)+B.\notag
\end{align}
We need to break $B$ into two pieces, one with $|u\otimes u|$ and the other with $|u\otimes v|$.  This is done with Minkowski's inequality then the triangle inequality by
\begin{align}
B&\leq r^2\int_{B(r)}\left(\int_0^te^{-|\xi|^2(t-s)}(|\widehat{v\otimes u}|+|\widehat{u\otimes u}|)\,ds\right)^2\,d\xi\notag\\
&\leq r^2\left(\int_0^t\left(\int_{B(r)}e^{-2|\xi|^2(t-s)}(|\widehat{v\otimes u}|+|\widehat{u\otimes u}|)^2\,d\xi\right)^{\frac{1}{2}}\,ds\right)^2\notag\\
&\leq r^2\left(\int_0^t\left(\int_{B(r)}e^{-2|\xi|^2(t-s)}(|\widehat{v\otimes u}|)^2\,d\xi\right)^{\frac{1}{2}}\,ds + \int_0^t\left(\int_{B(r)}e^{-2|\xi|^2(t-s)}(|\widehat{u\otimes u}|)^2\,d\xi\right)^{\frac{1}{2}}\,ds\right)^2.\notag
\end{align}
Using H\"older's inequality, then the decay assumption on $v$ (here $\eta\neq\infty$)
\begin{align}
\left(\int_{B(r)}e^{-2|\xi|^2(t-s)}(|\widehat{v\otimes u}|)^2\,d\xi\right)^{\frac{1}{2}}&\leq\left(\int_{B(r)}e^{-p|\xi|^2(t-s)}\,d\xi\right)^{\frac{1}{p}}\|\widehat{v\otimes u}\|_{L^q(\mathbb{R}^2)}\notag\\
&\leq C(t-s)^{-\frac{1}{p}}\|u(s)\|_{L^2(\mathbb{R}^2)}\|v(s)\|_{L^p(\mathbb{R}^2)}\notag\\
&\leq C(t-s)^{-\frac{1}{p}}\|u(s)\|_{L^2(\mathbb{R}^2)}s^{-(\frac{1}{2}-\frac{1}{p})}\notag
\end{align}
where in the above sequence $\frac{1}{2}=\frac{1}{q}+\frac{1}{p}$.  Also,
\begin{align}
\left(\int_{B(r)}e^{-2|\xi|^2(t-s)}(|\widehat{u\otimes u}|)^2\,d\xi\right)^{\frac{1}{2}} 
&\leq C|r|\|\widehat{u\otimes u}\|_{L^\infty(\mathbb{R}^2)}\notag\\
&\leq C|r|\|u(s)\|_{L^2(\mathbb{R}^2)}^2\notag
\end{align}
so that 
\begin{align}
B&\leq Cr^2\left(\int_0^t(t-s)^{-\frac{1}{p}}\|u(s)\|_{L^2(\mathbb{R}^2)}s^{-(\frac{1}{2}-\frac{1}{p})}\,ds\right)^2\notag\\
&\ \ \ \ \ \ \ \ +Cr^4\left(\int_0^t\|u(s)\|^2_{L^2(\mathbb{R}^2)}\,ds\right)^2.\notag
\end{align}
Then (\ref{energysplit}) becomes
\begin{align}
\frac{1}{2}\frac{d}{dt}\|u\|_{L^2(\mathbb{R}^2)}^2+ &(r^2-C_0(t+1)^{-1})\|u\|_{L^2(\mathbb{R}^2)}^2\notag\\
&\leq r^2I(t) +Cr^4\left(\int_0^t(t-s)^{-\frac{1}{p}}\|u(s)\|_{L^2(\mathbb{R}^2)}s^{-(\frac{1}{2}-\frac{1}{p})}\,ds\right)^2\notag\\
&\ \ \ \ \ \ \ \ +Cr^6\left(\int_0^t\|u(s)\|^2_{L^2(\mathbb{R}^2)}\,ds\right)^2.\notag
\end{align}
Choose $r^2(t)=\frac{1+C_0}{(t+1)}$ and multiply everything by $2(t+1)^2$ to find
\begin{align}
\frac{d}{dt}\left((1+t)^2\|u\|_{L^2(\mathbb{R}^2)}^2\right)&\leq C(t+1)I(t)\notag\\
&\ \ \ \ +C\left(\int_0^t(t-s)^{-\frac{1}{p}}\|u(s)\|_{L^2(\mathbb{R}^2)}s^{-(\frac{1}{2}-\frac{1}{p})}\,ds\right)^2\notag\\
&\ \ \ \ \ \ \ \ +C(1+t)^{-1}\left(\int_0^t\|u(s)\|^2_{L^2(\mathbb{R}^2)}\,ds\right)^2.\notag
\end{align}
By assumption we have
\begin{align}
I(t)\leq C(1+t)^{-\gamma}\label{gammaass}
\end{align}
for some $\gamma\in [0,1]$.  The next step is to integrate from $t_0$ to $\rho$ and divide by $(1+\rho)^{2-\gamma}$, which leads to
\begin{align}
(1+\rho)^{\gamma}&\|u(\rho)\|_{L^2(\mathbb{R}^2)}^2\notag\\
&\leq \frac{(1+t_0)^2}{(1+\rho)^{2-\gamma}}\|u(t_0)\|_{L^2(\mathbb{R}^2)}^2+\frac{C}{(1+\rho)^{2-\gamma}}\int_0^\rho(t+1)I(t)\,dt+A_1+A_2,\notag\\
A_1&=\frac{C}{(1+\rho)^{2-\gamma}}\int_0^\rho\left(\int_0^t(t-s)^{-\frac{1}{p}}\|u(s)\|_{L^2(\mathbb{R}^2)}s^{-(\frac{1}{2}-\frac{1}{p})}\,ds\right)^2\,dt.\notag\\
A_2&=\frac{C}{(1+\rho)^{2-\gamma}}\int_0^\rho(1+t)^{-1}\left(\int_0^t\|u(s)\|^2_{L^2(\mathbb{R}^2)}\,ds\right)^2\,dt.\notag
\end{align}
The main goal now is to set it up as a Gronwall inequality for $g(\rho)=(1+\rho)^\gamma\|u(\rho)\|_2^2$. For the $A_1$ term we have
\begin{align}
\int_0^t&(t-s)^{-\frac{1}{p}}\|u(s)\|_{L^2(\mathbb{R}^2)}s^{-(\frac{1}{2}-\frac{1}{p})}\,ds\notag\\
&\leq \left(\int_0^t(t-s)^{-\frac{2}{p}}s^{-(1-\frac{2}{p})}(1+s)^{-1}\,ds\right)^{\frac{1}{2}}\left(\int_0^t(1+s)\|u(s)\|^2_{L^2(\mathbb{R}^2)}\,ds\right)^{\frac{1}{2}}\notag\\
&\leq C\left(\int_0^\rho(1+s)\|u(s)\|^2_{L^2(\mathbb{R}^2)}\,ds\right)^{\frac{1}{2}}.\notag
\end{align}
Here we used $\int_0^t(t-s)^{-\frac{1}{p}}s^{-(\frac{1}{2}-\frac{1}{p})}\,ds<C$ for all $t>0$ when $p>2$.
Then,
\begin{align}
A_1&\leq \frac{C}{(1+\rho)^{2-\gamma}}\left(\int_0^\rho\,dt\right)\left(\int_0^\rho(1+s)\|u(s)\|^2_{L^2(\mathbb{R}^2)}\,ds\right)\notag\\
&\leq C\int_0^\rho(1+s)\|u(s)\|^2_{L^2(\mathbb{R}^2)}\,ds.\notag
\end{align}
In moving to the last line we used the fact that $\gamma\leq 1$. The $A_2$ term is similar, as
\begin{align}
A_2&=\frac{C}{(1+\rho)^{2-\gamma}}\int_0^\rho(1+t)^{-1}\left(\int_0^t\|u(s)\|^2_{L^2(\mathbb{R}^2)}\,ds\right)^2\,dt\notag\\
&\leq\frac{C}{(1+\rho)^{2-\gamma}}\int_0^\rho(1+t)^{-1}\left(\int_0^t(1+s)^{-1}\|u(s)\|^2_{L^2(\mathbb{R}^2)}\,ds\right)\left(\int_0^t(1+s)\|u(s)\|^2_2\,ds\right)\,dt\notag\\
&\leq\frac{C}{(1+\rho)^{2-\gamma}}\left(\int_0^\rho(1+t)^{-1}\int_0^t(1+s)^{-1}\|u(s)\|^2_{L^2(\mathbb{R}^2)}\,ds\,dt\right)\left(\int_0^\rho(1+s)\|u(s)\|^2_2\,ds\right).\notag
\end{align}
Now, by the assumed bound (\ref{initialenergy}), $\|u(s)\|_{L^2(\mathbb{R}^2)}^2\leq C(1+s)$ so that
\begin{align}
A_2\leq C\int_0^\rho(1+s)\|u(s)\|^2_{L^2(\mathbb{R}^2)}\,ds.\notag
\end{align}
The term $\frac{1}{(1+\rho)^{2-\gamma}}\|u(t_0)\|_2^2$ is bounded by some constant.  Using the assumption on $I(t)$
\begin{align}
\frac{C}{(1+\rho)^{2-\gamma}}\int_0^\rho(t+1)I(t)\,dt\leq \frac{C}{(1+\rho)^{2-\gamma}}\int_0^\rho(t+1)^{1-\gamma}\,dt\leq C\notag
\end{align}
Putting everything together we have
\begin{align}
g(\rho)&\leq C+C\int_0^\rho g(s)\,ds\notag\\
g(\rho)&=(1+\rho)^\gamma\|u(\rho)\|_2^2,\notag
\end{align}
so Gronwall's inequality implies $g(\rho)\leq C$ or
\begin{align}
\|u(t)\|_2^2\leq C(1+t)^{-\gamma},\notag
\end{align}
This is exactly the conclusion in Theorem \ref{maintheorem}.

\bibliography{NSdecay}

\bibliographystyle{unsrt}

\end{document}